\documentclass[a4paper,oneside]{amsart}

\usepackage{amssymb}
\usepackage{amsmath}
\usepackage{amsfonts}
\usepackage{amsthm}
\usepackage{mathrsfs}
\usepackage{fix-cm}
\usepackage{graphicx}
\usepackage{amscd}
\usepackage{enumerate}
\usepackage{rotating}
\usepackage{url}
\usepackage{wasysym}

\usepackage{soul}
\sodef\spred{}{.2em}{.9em plus.4em}{1em plus.1em minus.1em}

\usepackage[latin1]{inputenc}
\usepackage[T1]{fontenc}

\usepackage[english]{babel}

\usepackage{bbm}

\usepackage[all]{xy}

\usepackage{hyperref}

\def\intl{\int\limits}
\def\supl{\sup\limits}
\def\oplusl{\mathop\oplus\limits}
\def\amalgl{\mathop\amalg\limits}

\def\cupl{\mathop\cup\limits}

\def\one{{1\hspace{-0.1cm}\rm I}}
\def\zero{\mathbb{O}}

\newcommand{\bin}[2]{\binom{#1}{#2}}
\newcommand\isom{\xrightarrow{\,\smash{\raisebox{-0.65ex}{\ensuremath{\sim}}}\,}}

\newtheorem{Theorem}{Theorem}[section]
\newtheorem{Corollary}[Theorem]{Corollary}
\newtheorem{Lemma}[Theorem]{Lemma}
\newtheorem{Proposition}[Theorem]{Proposition}

\theoremstyle{definition}

\newtheorem{Remark}[Theorem]{Remark}
\newtheorem{Example}[Theorem]{Example}
\newtheorem{Definition}[Theorem]{Definition}

\newbox\mybox
\def\overtag#1#2#3{\setbox\mybox\hbox{$#1$}\hbox to
  0pt{\vbox to 0pt{\vglue-#3\vglue-\ht\mybox\hbox to \wd\mybox
      {\hss$\ss#2$\hss}\vss}\hss}\box\mybox}
\def\undertag#1#2#3{\setbox\mybox\hbox{$#1$}\hbox to 0pt{\vbox to
    0pt{\vglue#3\vglue\ht\mybox\hbox to \wd\mybox
      {\hss$\ss#2$\hss}\vss}\hss}\box\mybox}
\def\lefttag#1#2#3{\hbox to 0pt{\vbox to 4pt{\vss\hbox to
      0pt{\hss$\ss#2$\hskip#3}\vss}}#1}
\def\blefttag#1#2#3{\hbox to 0pt{\vbox to 4pt{\vss\hbox to
      0pt{\hss$#2$\hskip#3}\vss}}#1}
\def\righttag#1#2#3{\hbox to 0pt{\vbox to 4pt{\vss\hbox to
      0pt{\hskip#3$\ss#2$\hss}\vss}}#1}
\let\ss\scriptstyle

\def\Dot{\lower.2pt\hbox to 3.5pt{\hss$\bullet$\hss}}
\def\Circ{\lower.2pt\hbox to 3.5pt{\hss$\circ$\hss}}

\def\splicediag#1#2{\xymatrix@R=#1pt@C=#2pt@M=0pt@W=0pt@H=0pt}

\newcommand\lineto{\ar@{-}}
\newcommand\dashto{\ar@{--}}
\newcommand\dotto{\ar@{.}}

\newcommand{\C}{\mathbb{C}}
\newcommand{\R}{\mathbb{R}}
\newcommand{\K}{\mathbbm{K}}

\newcommand{\norm}[1]{\lVert #1 \rVert}
\newcommand{\num}[1]{\lvert #1 \rvert}
\newcommand{\morf}[4][\to]{ #2 \colon #3 #1 #4}
\newcommand{\inv}{^{-1}}

\newcommand{\rank}{\operatorname{rank}}
\newcommand{\codim}{\operatorname{codim}}
\newcommand{\Gl}{\operatorname{GL}}
\newcommand{\im}{\operatorname{Im}}
\newcommand{\trace}{\operatorname{trace}}

\renewcommand{\phi}{\varphi}
\renewcommand{\epsilon}{\varepsilon}

\begin{document}

\bibliographystyle{alpha}

\title{Lipschitz Normal Embeddings in the Space of Matrices}
\author{Dmitry Kerner}
\address{Department of Mathematics,\\ Ben Gurion University of
  Negev,\\ Israel}
\author{Helge M\o{}ller Pedersen}
\address{Dpartamento de Matem\'atica, ICMC\\
Universidade de S\~ao Paulo\\
134560-970 S\~ao Carlos, S.P.\\
Brazil}
\author{Maria A. S. Ruas}
\address{Dpartamento de Matem\'atica, ICMC\\
Universidade de S\~ao Paulo\\
134560-970 S\~ao Carlos, S.P.\\
Brazil}
\email{kernerdm@math.bgu.ac.il}
\email{helge@imf.au.dk}
\email{maasruas@icmc.usp.br}
\begin{abstract} 
The germ of an algebraic variety is naturally equipped with two
different metrics up to bilipschitz equivalence. The inner metric and
the outer metric. One calls a germ of a variety Lipschitz normally
embedded if the two metrics are bilipschitz equivalent. In this
article we prove Lipschitz normal embeddedness of some algebraic subsets of
the space of matrices. These include the space $m\times n$ matrices,
symmetric matrices and skew-symmetric matrices of rank  equal to a given number and their closures, and the upper triangular matrices with determinant $0$. We
also make a short discussion about generalizing these results to determinantal varieties in real and complex spaces. 
\end{abstract}
\maketitle

\section{Introduction}

If $(X,0)$ is the germ of an algebraic (analytic) variety over $\K =
\R$ or $\C$, then one
can define two natural metrics on it. Both are defined by choosing an
embedding of $(X,0)$ into $(\K^N,0)$. The first is the \emph{outer
	metric}, where the distance between two points $x,y\in X$ is given by
$d_{out}(x,y) := \norm{x-y}_{\K^N}$, i.e. the restriction of the
Euclidean metric to $(X,0)$. The other is the \emph{inner metric},
where the distance is defined as
\begin{align}
d_{in}(x,y) := \inf_{\gamma}
\big{\{} length_{\K^N}(\gamma)\ \big{\vert}\ \morf{\gamma}{[0,1]}{X} \text{
	 rectifiable, } \gamma(0) = x,\
\gamma(1) = y \big{\}}.\label{innerdefn} 
\end{align}
Both of these metrics are independent of the
choice of the embedding up to bilipschitz equivalence. The outer metric
determines the inner metric, and it is clear that $d_{out}(x,y) \leq
d_{in} (x,y)$. The other direction is in general not true, and one says
that $(X,0)$ is \emph{Lipschitz normally embedded} if the inner and
outer metrics are bilipschitz equivalent. \emph{Bilipschitz geometry}
is the study of the bilipschitz equivalence classes of these two
metrics. 

The study of bilipschitz geometry of complex spaces
started with Pham and Teissier who studied the case of curves in
\cite{phamteissier}. It then lay dormant for long time until Birbrair
and Fernandes began studying the case of complex surfaces
\cite{birbrairfernandes}. Among important recent results are the
complete classification of the inner metrics of surfaces by Birbrair,
Neumann and Pichon \cite{thickthin}, the proof that Zariski
equisingularity is equivalent to bilipschitz triviality in the case of
surfaces by Neumann and Pichon \cite{zariski} and the proof that
outer Lipschitz regularity implies smoothness by Birbrair,
Fernandes, L\^{e} and Sampaio
\cite{lipschitzregularity}. 

Understanding the geometry of the model varieties in the space of
matrices is an important step in understanding determinantal
singularities in real and complex spaces. We will also give a brief
discussion of this.

Determinantal singularities is also an area
that has been around for a long time, that recently saw a lot of
interest. They can be seen as a generalization of isolated complete
intersections (ICIS for short), and the recent
results have mainly been in the study of invariants coming from their
deformation theory. In \cite{ebelingguseinzade} \'Ebeling and
Guse{\u\i}n-Zade defined the index of a $1$-form, and the Milnor
number has been defined in various different ways by Ruas and da
Silva Pereira \cite{cedinhamiriam}, Damon and Pike \cite{damonpike}
and Nu\~no-Ballesteros, Or\'efice-Okamoto and Tomazella
\cite{NunoOreficeOkamotoTomazella}. Their
deformation theory has also been studied by Gaffney and Rangachev
\cite{gaffenyrangachev} and Fr\"uhbis-Kr\"uger and Zach
\cite{fruhbiskrugerzach}.

In January 2016 Asuf Shachar asked the following question on
Mathoverflow.org (\url{http://mathoverflow.net/questions/222162}): Is
the Lie group $\Gl_n^+(\R)$ Lipschitz normally
embedded, where $\Gl_n^+(\R)$ is the group of $n\times n$ matrices
with positive determinants. A positive answer  was given by the first
author and Katz, Katz and Liokumovich
in \cite{kerneretc}. They first prove it for ${X_{n-1}}$ the
set of $n\times n$  matrices with  rank $n-1$ and for its closure $\overline{X_{n-1}},$ the set of  matrices with determinant
equal to zero. Then they replace the segments of the straight line
between two points of $\Gl_n^+(\R)$ that passes trough $\Gl_n^-(\R)$ with a curve
arbitrarily close to $\overline{X_{n-1}}$. Their proof relies on topological
arguments, and some results on conical stratifications of MacPherson
and Procesi \cite{macphersonprocesi}. In this article we give an
alternative proof relying only on linear algebra and simple
trigonometry, which also works for $m\times n$ matrices of rank equal  to $t\leq \min\{m,n\}$ and their closures. (A first version of this proof appeared
in \cite{lnedeterminantalsing}). We  also prove the Lipschitz normal
embeddedness of  the symmetric
and skew-symmetric matrices of rank equal to a given $t$ and their closures, the
upper triangular matrices which have determinant $0$, and the
intersections with linear subspaces transversal to the rank stratification.

This article is organized as follows. In section \ref{preliminaries}
we discuss the basic notions of Lipschitz normal embeddings and give
some results concerning when a space is Lipschitz normally
embedded. In section \ref{geometryofmatrices} we describe the basic
properties of the bilipschitz geometry of the spaces of matrices we
consider. In section \ref{modelcase} we prove that the set ${X_{t}}$ of
matrices, symmetric matrices and skew-symmetric matrices of rank equal
to a given $t$ and their corresponding closures $\overline{X_{t}}$  are Lipschitz normally embedded, and that the same is
true if $V$ is a linear subspace transverse to the rank
stratification. We prove that the space of upper triangular matrices
with determinant $0$ is Lipschitz normally embedded in section
\ref{uppertriangularcase}. Finally in section \ref{secgeneralcase} we
discuss some of the difficulties to extend these results to the setting
of general determinantal singularities.


\section{Preliminaries on bilipschitz geometry }\label{preliminaries}

In this section we discuss some properties of Lipschitz normal
embeddings.
\begin{Definition}
We say that $X$ is \emph{Lipschitz normally embedded} if there exist
$K\geq 1$ such that for all $x,y\in X$,
\begin{align}
d_{in}(x,y)\leq Kd_{out}(x,y).\label{lneeq}
\end{align}
We call a $K$ that satisfies the inequality \emph{a bilipschitz
  constant of} $X$.
\end{Definition}

A trivial example of a Lipschitz normally embedded set is $\C^n$. For
an example of a space that is not Lipschitz normally embedded,
consider the plane curve given by $x^3-y^2=0$, then
$d_{out}((t^2,t^3),(t^2,-t^3))=2\num{t}^{3}$ but the
$d_{in}((t^2,t^3),(t^2,-t^3))= 2\num{t}^2+ o(t^2)$, this implies that
$\tfrac{d_{in}((t^2,t^3),(t^2,-t^3))}{d_{out}((t^2,t^3),(t^2,-t^3))}$
is unbounded as $t\to 0$, hence there cannot exist a $K$
satisfying \eqref{lneeq}.

Pham and Teissier \cite{phamteissier} show that in general the outer
geometry of a complex plane curve is equivalent to its embedded
topological type, and the inner geometry is equivalent to the abstract
topological type. Hence a plane curve is Lipschitz normally embedded
if and only if it is a union of smooth curves intersecting
transversely. See also Fernandes \cite{fernandesplanecurve} and
Neumann and Pichon \cite{NeumannPichonBilpischitzGeometryOfCurves}.

In the cases of higher dimension the question of which singularities
are Lipschitz normally embedded becomes much more complicated. It is no
longer only rather trivial singularities that are Lipschitz normally
embedded, for example in the case of surfaces the second author
together with Neumann and Pichon, shows that rational surface
singularities are Lipschitz normally embedded if and only if they are
minimal \cite{normallyembedded}. As we will later see, singularities
in the space of matrices give examples of non-trivial Lipschitz normally
embedded singularities in arbitrary dimensions.

\begin{Remark}
A couple of remarks about notation. Throughout the article $\K$ will
always denote $\R$ and $\C$. We will often be talking about different
inner distances of two points $x,y\in \K^N$, when we consider $x,y$ as
lying in different subspaces, hence $d_{in}^V(x,y)$ is the inner
distance between $x$ and $y$ measured using the inner metric on the
subspace $V\subset\K^N$. When we are using different outer metrics we
also denote the outer distance measured in $V$ by $d_{out}^V(x,y)$. 
\end{Remark}

First we explore the relationship between being Lipschitz
normally embedded local and being it global.
\begin{Definition}
A space $X$ is \emph{locally Lipschitz normally embedded} at $x\in X$
if there is an open neighbourhood $U$ of $x$, such that
$U$ is Lipschitz normally embedded. We say that $X$ is locally
Lipschitz normally embedded if this condition holds for all $x\in X$.
\end{Definition}

It is clear that being Lipschitz normally embedded implies being
locally Lipschitz normally embedded. In the other direction we have:

\begin{Proposition}\label{localglobal}
Let $X$ be a connected, compact locally Lipschitz normally embedded
space. Then $X$ is Lipschitz normally embedded.
\end{Proposition}

\begin{proof}
For each $x\in X$ let $U_x$ be a Lipschitz normally embedded
neighbourhood of $x$, and let $K_x$ be a bilipschitz constant. This implies
that if $y\in X$ is very close to $x$, then $d_{in}(x,y)\leq K_x
  d_{out}(x,y)$. Consider the map
\begin{align*}
\morf{f(x,y):= \frac{d_{in}(x,y)}{d_{out}(x,y)}}{ M\times M}{\R}.
\end{align*}
Let $U\subset M\times M$ be a small open tubular neighbourhood of the
diagonal $\Delta$. Then $f$ is continuous on the compact set $(M\times
M)\setminus U$ and locally bounded at each point. Thus it is globally
bounded on $(M\times
M)\setminus U$ and also on $U$.
\end{proof}
A simple consequence of this is the following.
\begin{Corollary}
Let $M$ be a connected compact manifold, then $M$ is Lipschitz
normally embedded.
\end{Corollary}

We will next give some results about when spaces constructed from
Lipschitz normally embedded spaces are themselves Lipschitz normally
embedded. First is the case of product spaces.

\begin{Proposition}\label{product}
Let $X\subset \K^n$ and $Y\subset \K^m$ and let $Z = X\times Y
\subset \K^{n+m}$. $Z$ is Lipschitz normally embedded if and only if
$X$ and $Y$ are Lipschitz normally embedded.
\end{Proposition}

\begin{proof}
First we prove the ``if'' direction.
Let $(x_1,y_1),(x_2,y_2)\in X\times Y$. We need to show that
\begin{align*}
d_{in}^{X\times Y}((x_1,y_1)(x_2,y_2))\leq K d_{out}^{X\times
  Y}((x_1,y_1)(x_2,y_2)).
\end{align*}
 Let $K_X$ be the constant such that
$d_{in}^X(a,b) \leq K_X d_{out}^X(a,b)$ for all $a,b\in X$, and let $K_Y$
be the constant such that $d_{in}^Y(a,b) \leq K_Y d_{out}(a,b)^Y$ for all
$a,b\in Y$. We get, using the triangle inequality, that
\begin{align*}
d_{in}^{X\times Y}((x_1,y_1)(x_2,y_2))\leq d_{in}^{X\times
  Y}((x_1,y_1)(x_1,y_2))+ d_{in}^{X\times Y}((x_1,y_2)(x_2,y_2)). 
\end{align*}
Now the points $(x_1,y_1)$ and $(x_1,y_2)$ both lie in the slice
$\{x_1\}\times Y$ and hence $d_{in}^{X\times
  Y}((x_1,y_1)(x_1,y_2)) \leq d_{in}^{Y}(y_1,y_2)$ and likewise we have
  $d_{in}^{X\times 
  Y}((x_1,y_2)(x_2,y_2)) \leq d_{in}^{X}(x_1,x_2)$. This then implies that
\begin{align*}
d_{in}^{X\times Y}((x_1,y_1)(x_2,y_2))\leq K_Y d_{out}^{Y}(y_1,y_2)+
K_X d_{out}^{X}(x_1,x_2),
\end{align*}
where we use that $X$ and $Y$ are Lipschitz normally embedded. Now it
is clear that $d_{out}^{X\times Y}((x_1,y_1)(x_1,y_2)) =
d_{out}^{Y}(y_1,y_2)$ and $d_{out}^{X\times Y}((x_1,y_2)(x_2,y_2)) =
d_{out}^{X}(x_1,x_2)$. Also, since 
\begin{align*}
d_{out}^{X\times
  Y}((x_1,y_1)(x_2,y_2))^2=d_{out}^{Y}(y_1,y_2)^2+
d_{out}^{X}(x_1,x_2)^2
\end{align*}
 by definition of the product metric, we have
that 
\begin{align*}
&d_{out}^{X\times Y}((x_1,y_1)(x_1,y_2)) \leq d_{out}^{X\times
  Y}((x_1,y_1)(x_2,y_2)) \text{ and } \\ &d_{out}^{X\times
  Y}((x_1,y_2)(x_2,y_2)) \leq d_{out}^{X\times
  Y}((x_1,y_1)(x_2,y_2)).
\end{align*}
 It then follows that
\begin{align*}
d_{in}^{X\times Y}((x_1,y_1)(x_2,y_2))\leq (K_Y + K_X) d_{out}^{X\times
  Y}((x_1,y_1)(x_2,y_2)).
\end{align*} 
For the other direction, let $p,q \in X$ and consider any path
$\morf{\gamma}{ [0,1] }{Z}$ such that $\gamma(0) = (p,0)$ and
$\gamma(1) = (q,0)$. Now $\gamma(t)= \big(\gamma_X(t),\gamma_Y(t)\big)$ where
$\morf{\gamma_X}{ [0,1] }{X}$ and $\morf{\gamma_Y}{ [0,1] }{Y}$ are
paths and $\gamma_X(0) = p$ and $\gamma_X(1) =q$. Now $l(\gamma)\geq
l(\gamma_X)$, hence 
\begin{align*}
d_{in}^X(p,q) \leq d_{in}^Z((p,0),(q,0)). 
\end{align*}
Since
$Z$ is Lipschitz normally embedded, there exist a $K>1$ such that
$d_{in}^Z(z_1,z_2)\leq K d_{out}(z_1,z_2)$ for all $z_1,z_2\in Z$. We
also have that $d_{out}^Z((p,0),(q,0))= d_{out}^X (p,q)$, since $X$ is
embedded in $Z$ as $X\times \{ 0\}$. Hence 
\begin{align*}
d_{in}^X(p,q) \leq K
d_{out}^X (p,q). 
\end{align*}
The argument for $Y$ being Lipschitz normally
embedded is the same exchanging $X$ with $Y$.
\end{proof}

\begin{Proposition}\label{bilipschizttrivial}
Let $X =\cup X_r \subset \K^n$ be a locally Lipschitz sratification
(see Parusi\'nski \cite{parusinski} Definition 1.1), and
assume that $X$ is Lipschitz normally embedded. Let $V$ be a $C^1$
manifold and let $x\in V\cap X$, $x\in X_r$. Assume that there exist an open
neighbourhood $U$ of $x$ such that for all $y\in U\cap X$, $y\in
X_{r(y)}$, we have that $V$ is transverse to $X_{r(y)}$ at $y$. Then
$V\cap X$ is locally Lipschitz normally embedded at $x$.
\end{Proposition}

\begin{proof}
Since $V$ is transverse to $X_{r(y)}$ at all $y\in U\cap X$, we can
(maybe by shrinking $U$) choose a map $\morf{\rho}{U}{X_r\cap U}$
which is a proper submersion restricted to each stratum, such that
$\rho\inv(x)=V\cap U$. By
the Lipschitz isotopy lemma (Theorem 1.9 in \cite{parusinski}) there
exist a bilipschitz trivilization  $\morf{\phi}{U}{U_S\times U_T}
$ of $X$, where $U_S\subset \K^{\dim(X_r)}$ and $U_T\subset
\K^{\codim(X_r)}$, such that the following diagram commutes:
\begin{align*}
\xymatrix{
X\cap U \ar[rr]^{\phi} \ar[rd]^{\rho} && \rho\inv(x)\times (X_r\cap
  U) \ar[dl]_{\pi} \\
& X_r\cap U,
}
\end{align*}
where $\pi$ is just the projection to the second factor. Now $\phi$ is
a bilipschitz map so $\rho\inv(x)\times (X_r\cap U)$ is Lipschitz
normally embedded since $X\cap U$ is. Then we have by Proposition
\ref{product} that $\rho\inv(x)=V\cap U$ is Lipschitz normally
embeded. 

If $\dim(V)>\codim(X_r)$ then if $V_S=V\cap (X_r\cap U)$ we have
that $V\cap U$ is bilipschitz equivalent to $V_T\times V_S$. Now
$\dim(V_T)=\codim(X_r)$, so we can choose $\rho$ as above such that
$\rho\inv(x)=V_T$. Hence $V_T\cap X$ is Lipschitz normally embedded and
since $V_S$ is $C^1$ equivalent to $\K^{\dim(V)-\codim(X_r)}$ it is
also Lipschitz normally embedded. Thus $V\cap (X\cap U)$ is
Lipschitz normally embedded by Proposition \ref{product}, since it is
bilipschitz equivalnet to $(V_T\cap X)\times V_S$.
\end{proof}

Another case we will need later is the case of cones.

\begin{Proposition}\label{cone}
Let $X\subset \K^n$ be the cone over $M\subset S$ with cone point
the origin of $\K^n$, where $S=S^{n-1}$ if $\K=\R$ and $S=S^{2n-1}$ if
$\K=\C$. Then the following conditions hold:
\begin{enumerate}[(a)]
\item If $M$ is Lipschitz normally embedded then $X$ is Lipschitz
  normally embedded.\label{b}
\item If $X$ is Lipschitz normally embedded and $M$ is compact, then
  each of the connected components of $M$ is Lipschitz normally
  embedded.\label{a}
\end{enumerate} 
\end{Proposition}
\begin{proof}
We first prove \eqref{b}. Since $M$ is Lipschitz
normally embedded with bilipschitz constant $K_M$ the same is true for
$r\cdot M=rM$, where $r\in\R^+$.

Let $x,y\in X$. We can assume that $0\leq \norm{x} \leq \norm{y}$. If
$x=0$ then $d_{in}^{X}(x,y)=d_{out}(x,y)$ since the straight
line through $0$ and $y$ is in $X$ because $X$ is conical. 

If $\norm{x}=\norm{y}=r$, then $x$ and $y$ are both in $rM$, and hence
\begin{align*}    
d_{in}^X(x,y) \leq d_{in}^{rM}(x,y) \leq K_Md_{out}(x,y).
\end{align*}

Now if $0<\norm{x}<\norm{y}$ let
$y'=\tfrac{y}{\norm{y}}\norm{x}$. Then $d_{in}^X(y,y')=d_{out}(y,y')$
since they both lie on the same straight line through the origin. If
$r=\norm{x}$, then $x,y'\in rM$. Hence like before $d_{in}^X(x,y')\leq
 K_Md_{out}(x,y')$.  Now $y'$ is the point
closest to $y$ in $rM$. Hence all of $rM$ lies on
the other side of the affine hyperplane through $y'$ orthogonal to the
line $\overline{yy'}$ from $y$ to $y'$. Hence the angle between
$\overline{yy'}$ and the line $\overline{y'x}$ between $y'$ and $x$ is
more than $\tfrac{\pi}{2}$. Therefore, the Euclidean distance from $y$ to
$x$ is larger than each of $l(\overline{yy'})$ and
$l(\overline{y'x})$. This gives us:
\begin{align*}
d_{in}^X(x,y) &\leq d_{in}^X(x,y')+ d_{in}^X(y',y) \leq
K_md_{out}(x,y')+d_{out}(y',y)\\ &\leq (K_m+1)d_{out}(x,y).
\end{align*}

To prove \eqref{a}, assume that $X$ is Lipschitz normally
embedded, but a connected component $M'\subset M$ is not Lipschitz
normally embedded. 

Since $M'$ is compact we can assume that $M'$ is not locally
Lipschitz normally embedded at some point by Proposition
\ref{localglobal}. So let $p\in M'$ be a point such that $M'$ is not
Lipschitz normally embedded in a small open neighbourhood $U\subset M'$
of $p$. By
Proposition \ref{product} we have that $U\times (-\epsilon,\epsilon)$ is not
Lipschitz normally embedded, where $0<\epsilon$ is much smaller than
the distance from $M$ to the origin. Now the quotient map from
$\morf{c}{M\times [0,\infty)}{X}$ induces an outer (and therefore also
inner) bilipschitz equivalence of $U\times (-\epsilon,\epsilon)$ with
$c\big(U\times (-\epsilon,\epsilon)\big)$. Since both $U$ and
$\epsilon$ can be chosen to be arbitrarily small, we have that there
does not exist any small open neighbourhood of $p\in X$ that is
Lipschitz normally embedded, contradicting that $X$ is Lipschitz
normally embedded. Hence $X$ being Lipschitz normally embedded implies
that $M'$ is Lipschitz normally embedded.  
\end{proof}

\begin{Remark}
\eqref{b} holds under the weaker hypothesis that $M$ has a finite
number of connected components each one being Lipschitz normally
embedded, and such that for each pair of connected components $X$ and
$Y$ we have $d_{out}(X,Y):=\inf_{x\in X,y\in Y}\{d_{out}(x,y)\}>0$.
If the number of connected components of $M$ is not finite, then the
result may fail as seen below.
  (In particular, it is not enough to ask that $M$ is locally compact, locally path-connected and locally Lipschitz normal.)
\begin{itemize}
\item Let $M=\cupl^\infty_{n=1}\{e^{\frac{\pi i}{n}}\}\subset S^1$. Thus $M$ is non-connected, non-compact, but (trivially) locally path-connected, locally compact, locally Lispchitz normal. But $Cone(M)\subset\R^2$ is not locally Lipschitz normal at the origin.
\end{itemize}
\end{Remark}

A consequence of Proposition \ref{cone} is the following.
\begin{Corollary}
Let $(X,0)$ be the germ of real or complex homogeneous variety with isolated
singularity, then $(X,0)$ is Lipschitz normally embedded.
\end{Corollary}

We conclude this section with a useful lemma.

Let $\phi \rightturn \R^N$ be a diffeomorphism in $\R^N.$ For each $x \in \R^N$ consider the Jacobian matrix $\frac{d\phi}{dx},$ it is non-degenerate. Let $\{\lambda_i(x)\}$ be its eigenvalues and fix $\lambda_{max}(x)=max||\lambda_i(x)||,$ $\lambda_{min}(x)=min||\lambda_i(x)||.$ Define
$$\lambda_{max}:=\, sup_{x\in \R^N} \lambda_{max}(x) \leq \infty,\,\,\,\,\,\,\,\,\,\, \lambda_{min}:=\, inf_{x\in \R^N} \lambda_{min}(x) \geq 0$$

\begin{Lemma}\label{changediffeo}
	For a diffeomorphism $\phi$ as above suppose $0 < \lambda_{min}$ and $\lambda_{max} < \infty.$ Let $X \subset \R^N$ be any path-connected subset. Then for any $x, y \in X$ the following holds: 
	$\lambda_{min}\, \cdot \,d_{in}^X(x,y)\leq d_{in}^{\phi(X)}(\phi(x), \phi(y))\leq \lambda_{max} \, \cdot  \, d_{in}^X(x,y).$

\end{Lemma}

\begin{proof}
	For fixed $x,y \in X$ choose a rectifiable path $\gamma \subset X$ connecting $x,y$ and satisfying: $length(\gamma) < d_{in}^{X}(x,y) + \epsilon.$ Then  
	$\phi(\gamma)\subset \phi(X)$ connects $\phi(x), \phi(y).$ It remains to compare 
	$length (\gamma)=\intl^1_0 \sqrt{||\dot \gamma(t)||^2}dt$ to 
	$$length (\phi(\gamma))=\intl^1_0 \sqrt{||\dot \phi(\gamma(t))||^2}dt=
	\intl^1_0 \sqrt{||\frac{d\phi}{dx}\cdot({\dot \gamma}(t))||^2}dt.$$
	
	Note that $\lambda_{min}\, \cdot \, ||\dot \gamma(t)|| 
	\leq ||\frac{d\phi}{dx}\cdot({\dot \gamma}(t))|| \leq \lambda_{max}\, \cdot \, ||\dot \gamma(t)||.$ Thus the bounds follow.
\end{proof}

\section{Geometry in the space of matrices}\label{geometryofmatrices}

Let $\K=\R$ or $\C$ and take the vector space of $m\times n$ matrices
over $\K$, $Mat_{m\times n}(\K)$, $1\leq m\le n$. We use
the standard inner product on $Mat_{m\times n}(\K)$, $\langle A,B
\rangle :=trace(A\overline{B^t})$, and the corresponding metric on
$Mat_{m\times n}(\K)\approx \K^{mn}$.

For any subset $X\subseteq Mat_{m\times n}(\K)$ consider the stratification by rank, $X_r:=X\cap\{A\in Mat_{m\times n}(\K)|\  rank(A)=r\}$. The strata $X_r$ are connected when $\K=\C,$ however when $\K= \R$  they may have various connected components.

Besides the   outer metric,
\begin{align*}
d_{out}(A,B)=\sqrt{\trace\Big((A-B)\cdot\overline{(A-B)^t}\Big)},
\end{align*}
the  sets $X_r$ have the inner metric, $d^{X_r}_{in}(A,B)$,
as defined in Equation \eqref{innerdefn} in the introduction. Similarly for the closures,
$\{\overline{X_r}\}$,
one has $d^{\overline{X_r}}_{in}(A,B)$.

Note that for some linear subspaces of $ Mat_{m\times n}(\K)$ the rank
stratification is not Lipschitz normally embedded as we will see in
Example \ref{degenerationofcusps}.

\subsection{The relevant group actions}\label{Sec.Group.Actions}
We use the action of two groups on $Mat_{m\times n}(\K)$ and on the
strata $\{X_r\}$.
\begin{itemize}
\item	Consider the group $U(m)=\{V \mid V\cdot\overline{V^t} =
  \one_{m\times m}\} \subset Mat_{m\times m}(\C)$ and similarly
  $U(n)\subset Mat_{n\times n}(\C)$. Their product acts, $U(m)\times
  U(n)\circlearrowright Mat_{m\times n}(\C)$, by $A\to V_lAV_r$. This
  group action is isometric, because we have that $\langle A,B\rangle
  \rightsquigarrow \langle V_lAV_r,V_lBV_r \rangle =trace(V_lAV_r\cdot
  \overline{(V_lBV_r)^t})=<A,B>$. For $\K=\R$ one takes the group $O(n)$.
	
 The group $U(n)$ is connected, thus if
$A\stackrel{U(n)}{\sim}B$ then there exists a path from $A$ to $B$,
given by the $U(n)$-action. The group $O(n)$ has two connected
components, in some cases we use the component $SO(n)$. 
	
	Given a  matrix $A\in X_r$, we can use the $U(n)$, $SO(n)$-
        action to bring the left and right kernels, $ker_l(A)$,
        $ker_r(A)$, to the form $(\underbrace{0,\dots,0}_{r},*,\dots,*)$. With
        this assumption $A$ becomes block-diagonal, hence we have that
\begin{align*}
 A\sim \begin{bmatrix} A_{inv} & \zero \\ \zero &
  \zero_{(m-r)\times(n-r)}\end{bmatrix},
\end{align*}
 here $A_{inv}\in Mat_{r\times r}(\K)$ is invertible.
	
\item
	Consider the group $\Gl_m\subset Mat_{m\times m}(\K)$ and
        similarly $\Gl_n\subset Mat_{n\times n}(\K)$. The product acts,
        $\Gl_m\times \Gl_n\circlearrowright Mat_{m\times n}(\K)$, by
        $A\to V_lAV_r$. This group action is not isometric. However,
        for any fixed pair $(V_l,V_r)$ the map $A\to V_lAV_r$ is
        bilipschitz as we see in Corollary \ref{chageofcoordinates} below.
	
	Moreover, the action preserves all the strata $\{X_r\}$ and
        acts on them transitively, e.g. any matrix $A\in X_r$ is
        equivalent to the canonical form, 
\begin{align*}
A\sim \begin{bmatrix}
          \one_{r\times r} &\zero\\ \zero &
          \zero_{(m-r)\times(n-r)} \end{bmatrix}.
\end{align*} 
 Therefore the
        tangent space to any of $X_r$, at any point $A$, can be
        computed as the tangent space to the orbit of $A$ under this
        group action.
\end{itemize}

The next result is an easy corollary of Lemma \ref{changediffeo}.

\begin{Corollary}\label{chageofcoordinates}
Let $V\subset Mat_{m,n}(\K)$ and $(C_l,C_r) \in \Gl_m\times\Gl_n.$  Then the map
$A\to C_lAC_r$ 
is a bilipschitz map from $V$ to $W=C_lVC_r$. In
particular if $A,B\in V$ satisfy $d_{in}(A,B)\leq Kd_{out}(A,B)$,
then $d_{in}(C_lAC_r, C_lBC_r) \leq K d_{out}(C_lAC_r,C_lBC_r)$
\end{Corollary}

\subsection{Connected components of the
  strata}\label{Sec.Connected.Components}
We first remark that in both  cases $\K=\R$ or $\C$, the sets
$\overline{X}_r$ are connected for all $r$ if $X$ is a linear subspace.

Let $\K=\C$ and $X$ be one of $Mat_{m\times n}(\C)$, $Mat^{sym}_{m\times m}(\C)$,
$Mat^{skew-sym}_{m\times m}(\C)$ or triangular matrices. Then all the
strata $X_r$ are connected. Indeed, they are all irreducible algebraic
varieties and thus $dim_{\C}(\overline{X_r})
-dim_\C(\overline{X}_{r-1})\ge1$, i.e. the complements are of real
codimension$\ge2$.

For $\K=\R$ the strata can have several connected components.
\begin{itemize}
\item Let $X=Mat_{m\times n}(\R)$, for $r=m=n$ we have the classical decomposition $X_n=\Gl^+_n(\R)\amalg \Gl^-_n(\R)$.
We prove that for $r<m$ the strata $X_{r}$  are connected. Indeed, given any $A\in X_r$ bring it to the block-diagonal form,
$A\stackrel{SO(m)\times SO(n)}{\sim}A_{inv}\oplus\zero$, as above. Here $A_{inv}$ is invertible and is defined up to $SO(r)\times SO(r)$
transformation. Thus for any $A,B\in X_r$ it is enough to connect $A_{inv}\oplus\zero$ to $B_{inv}\oplus\zero$.

If $det(A_{inv}B_{inv})>0$
then the two matrices are connected just inside $\Gl_r(\R)$. To address the case $det(A_{inv}B_{inv})<0$, it is enough to connect
$A_{inv}\oplus\zero$ to some  $\widetilde{A}_{inv}\oplus\zero$, with $det(A_{inv}\widetilde{A}_{inv})<0$.
We choose 
\begin{align*}
\widetilde{A}_{inv}=\begin{bmatrix}
  \one_{(r-1)\times(r-1)}&\zero\\\zero&-1_{1\times
    1}\end{bmatrix}\cdot A_{inv}
\end{align*}
 and construct the needed path as follows.
Choose any path $(x(t),y(t))$ from $(1,0)$ to $(-1,0)$ inside $\R^2\setminus\{(0,0)\}$, e.g. a half-circle. Let $V(t)\in \Gl^+_2(\R)$ be a
matrix family inducing this path, i.e. $\Bigl[\begin{smallmatrix}
  x(t)\\y(t)\end{smallmatrix}\Bigr] m=V(t)\Bigl[\begin{smallmatrix} 1\\0\end{smallmatrix}\Bigr]$, $V(0)=\one$ and $V(1)=\Bigl[\begin{smallmatrix} -1&0\\0&1\end{smallmatrix}\Bigr]$.
Accordingly consider the path
\begin{align*}
A(t)=\begin{bmatrix} \one_{(r-1)\times(r-1)}&\zero&\zero\\\zero& V(t)&\zero\\\zero&\zero&\zero\end{bmatrix}\cdot A
\end{align*}
By the construction $A(t)$ lies inside $X_r$ and connects
$A_{inv}\oplus\zero$ to $\widetilde{A}_{inv}\oplus\zero$. For $m<n$ all the
strata are connected by the similar argument.

\item
For $X=Mat^{sym}_{n\times n}(\R)$ and any $A\in X_r$ we have $A\stackrel{SO(n)}{\sim}A_{inv}\oplus\zero$, as before. Then use $SO(r)$
to diagonalize $A_{inv}$. The signs of the eigenvalues are preserved in continuous deformations inside $X_r$.
Therefore the decomposition into
the connected components is $X_r=\amalgl_{r_++r_-=r}\mathcal{U}_{r_+,r_-}$, where $\mathcal{U}_{r_+,r_-}\subset Mat^{sym}_{r\times r}(\R)$
is the subset of matrices of signature $(r_+,0,r_-)$.
\item
For $X=Mat^{skew-sym}_{n\times n}(\R)$ recall that the rank of a skew-symmetric matrix is always even, thus $X_{2r+1}=\varnothing$ and we work only with $X_{2r}$.
We prove that for $2r<n$ the stratum $X_{2r}$ is connected, while for $n$-even the stratum $X_n$ has two connected components.

Suppose $n$ is even, then  the canonical form under the $SO(n)$ action
is $\oplusl_i\begin{pmatrix} 0&\lambda_i\\-\lambda_i&0\end{pmatrix}$, and one can bring any matrix to this form in a continuous way. (Because $SO(n)$ is connected.)
Furthermore, if all $\lambda_i$ are non-zero,
then we can assume $\lambda_{i}>0$ for $i<n$. Indeed, the negative
$\{\lambda_i\}$ can be turned into positive in pairs by the $SO(n)$
transformation
\begin{align*} 
  &\begin{bmatrix}1&0\\0&-1\\&&1&0\\&&0&-1 \end{bmatrix} \begin{bmatrix}
    0&\lambda_i&0&0\\-\lambda_i&0&0&0\\0&0&0&\lambda_j\\0&0&-\lambda_j&0\end{bmatrix} 
  \begin{bmatrix} 1&0\\0&-1\\&&1&0\\&&0&-1 \end{bmatrix} \\
  &=\begin{bmatrix}  0&-\lambda_i&0&0\\\lambda_i&0&0&0\\0&0&0&-\lambda_j\\0&0&\lambda_j&0 
\end{bmatrix}.  
\end{align*}
(Note again that $SO(n)$ is connected.) Thus any canonical form is connected (inside $X_n$) to either $\oplusl_i \Bigl[\begin{smallmatrix} 0& 1\\ -1&0\end{smallmatrix}\Bigr]$ or to
$(\oplusl_i \Bigl[\begin{smallmatrix} 0&1\\-1&0\end{smallmatrix}\Bigr])\oplus \Bigl[\begin{smallmatrix} 0&-1\\1&0\end{smallmatrix}\Bigr]$. Finally we remark that the Pfaffian polynomial of a skew-symmetric  matrix, $Pf(A),$ is continuous under deformations of $A$ and $Pf|_{\mathcal U_{even}}>0,$ while $Pf|_{\mathcal U_{odd}}<0.$ Thus there are two connected components.
Therefore $X_{n}=\mathcal{U}_{even} \amalg \mathcal{U}_{odd}.$

For $X_{2r}$, with $2r<n$, we first use the equivalence $A\to V^tAV$, $V\in SO(n)$, to bring $A$ to the form $A_{inv}\oplus\zero$,
with $A_{inv}\in Mat_{2r\times 2r}^{skew-sym}(\R)$,
as in paragraph \ref{Sec.Group.Actions}. Then, as in the case of $X_n$, we bring $A_{inv}$ to either $\oplusl_i \Bigl[\begin{smallmatrix} 0& 1\\ -1&0\end{smallmatrix}\Bigr]$ or
$(\oplusl_i \Bigl[\begin{smallmatrix}
  0&1\\-1&0\end{smallmatrix}\Bigr])\oplus \Bigl[\begin{smallmatrix}
  0&-1\\1&0\end{smallmatrix}\Bigr]$. As $2r<n$, it remains to connect 
\begin{align*}
\begin{bmatrix} 0&1&0\\-1&0&0\\0&0&0\end{bmatrix} \text{ to }
\begin{bmatrix} 0&-1&0\\1&0&0\\0&0&0\end{bmatrix}.
\end{align*}
 This is done as in the case of $Mat_{m\times n}(\R)$. We fix a matrix family, $V(s)\in GL_2^+(\R)$,
that connects $(1,0)$ to $(-1,0)$ and consider the path
\begin{align}
\begin{bmatrix}1&\zero_{1\times 2}\\\zero_{2\times
    1}&V(s)\end{bmatrix}\begin{bmatrix}
  0&1&0\\-1&0&0\\0&0&0\end{bmatrix}\begin{bmatrix}1&\zero\\\zero&V(s)\end{bmatrix}^t=
\begin{bmatrix} 0 & v_{11} & v_{21}\\
-v_{11} & 0 & 0\\
-v_{21} & 0 & 0
\end{bmatrix}.
\end{align}
\end{itemize}

\subsection{The local structure of $\overline{X_r}$ and "controlled path-connectedness"}
In this section $\K\in\R,\C$ and we always consider small
neighbourhoods of spaces near some points. We freely use the germ
notation, e.g. $(\K^n,\zero)$ denotes
a small neighbourhood of $\K^n$ near the origin (i.e. near the zero matrix), $(\overline{X_r},A)$ denotes a small neighbourhood of
the matrix $A$ in $\overline{X_r}$, while $T_AX_r$
denotes the tangent space of $X_r$ at the point $A\in X_r$.

Sometimes to keep track of the size we denote the strata by  $X^{(m\times n)}_r$.

\begin{Lemma}\label{loacalalmostlne}
\begin{enumerate}[1.]
\item Let $X=Mat_{m\times n}(\K)$,
fix some $A\in\overline{X^{m\times n}_r}$, with $\rank(A)=r_0\le
r$. Then 
\begin{align*}
(\overline{X^{m\times n}_r},A)\approx (\K^{mr_0+nr_0-r^2_0},\zero)\times
(\overline{X^{(m-r_0)\times (n-r_0)}_{r-r_0}},\zero),
\end{align*}
where the homeomorphism is almost metric preserving, i.e. the metric
distortion can be assumed small if the germ representatives are
small.\label{321}
\item Similarly, for $X=Mat^{sym}_{m\times m}(\K)$  one has:
\begin{align*}
(\overline{X^{m\times m}_{r}},A)\approx
(\K^{mr_0-\bin{r_0}{2}},0)\times(\overline{X^{(m-r_0)\times
    (m-r_0)}_{r-r_0}},0),
\end{align*}
 while for
$X=Mat^{skew-sym}_{m\times m}(\K)$ one has:
\begin{align*}
(\overline{X^{m\times m}_{r}},A)\approx
(\K^{mr_0-\bin{r_0+1}{2}},0)\times(\overline{X^{(m-r_0)\times
    (m-r_0)}_{r-r_0}},0).
\end{align*}\label{322}
\end{enumerate}
\end{Lemma}
\begin{proof}
\eqref{321}. Using the linear isometries $U(m)\times U(n)$ we can assume the left/right kernels of $A$ in the form  $(\underbrace{0,\dots,0}_{r_0},*,\dots,*)$,
see paragraph \ref{Sec.Group.Actions}. Therefore $A=A_{inv}\oplus \zero_{(m-r_0)\times(n-r_0)}$, here
$A_{inv}\in Mat_{r_0\times r_0}(\K)$ is invertible.

As the action $\Gl_m\times \Gl_n\circlearrowright X_{r_0}$ is
transitive (and smooth) we write down the tangent space
$T_AX_{r_0}$ as the tangent to the orbit using
the calculation of the tangent space given in \cite{arbarellocornalbagriffiths}:
\begin{align*}
T_AX_{r_0}=Span_\R(V_lA,AV_r)_{\substack{V_l\in Mat_{m\times m}(\K)\\V_r\in Mat_{n\times n}(\K)}}=
Span_\K\Big(\begin{bmatrix} *&*&\\ *&\zero_{(m-r_0)\times(n-r_0)}\end{bmatrix}\Big)
\end{align*}
As the stratum $X_{r_0}$ is smooth (at any of its points), it can be rectified locally near $A$ to its tangent space. Namely, there exists a homeomorphism,
$(Mat_{m\times n}(\K),A)\approx (\K^{mr_0+nr_0-r^2_0},0)\times (\K^{(m-r_0)(n-r_0)},0)$, that sends $(X_{r_0},A)$
to $(T_AX_{r_0},0)\times\{0\}=(\K^{mr_0+nr_0-r^2_0},0)\times \{0\}$.
This homeomorphism is assured by the implicit function theorem and can be chosen "almost metric preserving". More precisely, for any $\epsilon>0$ the distortion
of the distances will be less than $\epsilon$ provided we choose a small enough neighbourhood of $A$ in $X_{r_0}$.

Restricting this homeomorphism to $(\overline{X^{m\times n}_r},A)$ we get the statement.

\eqref{322}. The proof is essentially the same, just here one uses the action
  $A\to V^t AV$, $V\in U(n)$.
\end{proof}

\begin{Lemma}\label{Thm.Path.Connected.Controlled}
\begin{enumerate}[1.]
\item Let $X=Mat_{m\times n}(\K)$. For any $r\le m\le n$ the connected
  components of $X_r$ are "controlled path-connected" near any point
  of $\overline{X_r}$ in the following sense:

for any $A\in \overline{X_r}$ and any $\epsilon>0$ there exists
$\delta=\delta(A,\epsilon)$ such that any points of the ball, $P,Q\in
Ball_\delta(A)\cap X_r$, belonging to the same connected component of
$X_r$, are connected (inside $Ball_\delta(A)\cap X_r$) by a path of
length$<\epsilon$.\label{331}

\item Similarly for the spaces of (skew-)symmetric matrices, $X=Mat^{sym}_{m\times m}(\K)$ or $X=Mat^{skew-sym}_{m\times m}(\K)$, their strata are
controlled path connected at any point.\label{332}
\end{enumerate}
\end{Lemma}
\begin{proof}
\eqref{331}. Let $\rank(A)=r_0\le r$, by the last lemma there exist
homeomorphisms as on the diagram. 
\begin{align*}
\begin{matrix} (Mat_{m\times n}(\K),A)&\stackrel{\phi}{\isom}&(\K^{mr_0+nr_0-r^2_0},\zero)\times Mat_{(m-r_0)\times (n-r_0)}(\K)
\\\cup&&\cup
\\(\overline{X_r},A)&\isom& (\K^{mr_0+nr_0-r^2_0},\zero)\times(\overline{X^{(m-r_0)\times (n-r_0)}_{r-r_0}},\zero)
\\\cup&&\cup
\\(X_r,A)&\isom&(\K^{mr_0+nr_0-r^2_0},\zero)\times(X^{(m-r_0)\times (n-r_0)}_{r-r_0},\zero).
\end{matrix}
\end{align*}
 Here in the last row we denote by $(X_r,A)$ a small
	neighbourhood of $X_r$ near $A$, even though $A\not\in X_r$. Similarly for $(X^{(m-r_0)\times (n-r_0)}_{r-r_0},\zero)$.
	
	While $\phi$ does not preserve the distances, the distortions are small for small representative, therefore it is
	enough to prove the statement for the presentation on the right.

Write the coordinates of $P,Q$ for this splitting, $P\rightsquigarrow(P_1,P_2)$, $Q\rightsquigarrow(Q_1,Q_2)$,
where $P_1,Q_1\in (\K^{mr_0+nr_0-r^2_0},\zero)$, while $P_2,Q_2\in (X^{(m-r_0)\times (n-r_0)}_{r-r_0},\zero)$.

Now take the paths $(tP_1,P_2)$, $(tQ_1,Q_2)$, where $t\in [0,1]$. Both paths lie inside
$(\K^{mr_0+nr_0-r^2_0},\zero)\times X^{(m-r_0)\times (n-r_0)}_{r-r_0}$, thus their pre-images under $\phi$ lie inside $X_r$.  And the lengths of
both paths are small for $\delta$ small. Therefore it remains to check the points $(0,P_2)$, $(0,Q_2)$, i.e. to connect them by
a short path that lies inside $\{0\}\times X^{(m-r_0)\times (n-r_0)}_{r-r_0}$.

By this transition we have reduced the problem from the case $P,Q,A\in \overline{X^{m\times n}_{r}}$ to the
case, $P_2,Q_2,\zero\in \overline{X^{(m-r_0)\times (n-r_0)}_{r-r_0}}$.
Note that $0<r-r_0\le m-r_0\le n-r_0$. Note that $P_2,Q_2$ still lie in the same connected component of $X^{(m-r_0)\times (n-r_0)}_{r-r_0}$, as the paths
are in $X_r$.

Thus we have to prove:

for any $\epsilon>0$ there exists $\delta=\delta(\epsilon)$ such that any points  $P,Q\in Ball_\delta(\zero)\cap X_r\subset Mat_{m\times n}(\K)$ are connected
(inside $X_r$) by a path of length$<\epsilon$.

Alternatively: {\em any point $P\in Ball_\delta(\zero)\cap X_r$ is connected to
	the special point $\delta\cdot\one_{r\times r}\oplus\zero_{(m-r)\times (n-r)}$ by a path of length$<\epsilon$.} And this later statement is immediate,
apply the Gauss elimination procedure on rows and columns (by $\Gl_m\times \Gl_n$) to get a path of bounded length.

\eqref{332}. For the (skew-)symmetric case the proof is essentially the same, just the special point is now
$\delta\cdot\one\oplus (-\delta\cdot\one)\oplus\zero_{(m-r)\times (n-r)}$ (the sizes depend on the signature) and
instead of the Gauss elimination one uses the action $A\to V^t AV$.
\end{proof}

\section{Lipschitz normality of linear subspaces of the space of
  matrices}\label{modelcase}

\subsection{Lipschitz normality for the closures $\overline{X_r}$}
\begin{Theorem}\label{Thm.Lipshitz.Normality.Closures.of.Strata}
Let $\K\in\R,\C$ and $X$ be one of the spaces $Mat_{m\times n}(\K)$, $Mat^{sym}_{n\times n}(\K)$, $Mat^{skew-sym}_{n\times n}(\K)$.
For any  $1\le r\le m\le n$ and $A,B\in\overline{X_r}$ holds:
$\frac{d_{in}^{\overline{X_r}}(A,B)}{2\sqrt{2}} \le d_{out}(A,B)\le
d^{\overline{X_r}}_{in}(A,B)$.
\end{Theorem}
\begin{proof}
The inequality on the right is immediate, we prove the one on the left.

We use the group action, $U(m)\times U(n)\circlearrowright Mat_{m\times n}(\K)$, by $A\to UAV$,
and $U(n)\circlearrowright Mat^{sym}_{n\times n}(\K)$,
$Mat^{skew-sym}_{n\times n}(\K)$, by $A\to U^tAU$, to bring $A$ to the
form
\begin{align*}
\begin{bmatrix} A_1&\zero_{r\times (n-r)}\\\zero_{(m-r)\times
    r}&\zero_{(m-r)\times(n-r)}\end{bmatrix}.
\end{align*}
 Here $A_1\in Mat_{r\times r}(\K)$, $Mat^{sym}_{r\times r}(\K)$,
$Mat^{skew-sym}_{r\times r}(\K)$. This action preserves $X_r$, $\overline{X_r}$ and the inner/outer distances. Therefore we can assume $A$ in this form.
Present $B$ accordingly: $\Bigl[\begin{smallmatrix} B_1&B_2\\B_3&B_4\end{smallmatrix}\Bigr]$. Then:
\begin{align*} d_{out}(A,B)=\sqrt{||A_1-B_1||^2+||B_2||^2+||B_3||^2+||B_4||^2}
\end{align*}
This is the distance along the straight segment. We will replace this straight segment by two parts, lying inside $\overline{X_r}$,
whose total length is less than $2d_{out}(A,B)$

Consider the path $B(t)=\Bigl[\begin{smallmatrix} B_1&tB_2\\ tB_3&t^2B_4\end{smallmatrix}\Bigr]$ for $t\in[0,1]$. We claim: $B(t)\in \overline{X_r}$ for any $t\in[0,1]$. Indeed, scaling a particular
row/column does not increase the rank. And in the (skew-)symmetric case $B(t)$ remains (skew-)symmetric.

Therefore we get an algebraic curve (inside $\overline{X_r}$) that connects $B=B(1)$ to $B(0)=\Bigl[\begin{smallmatrix} B_1&\zero\\\zero&\zero\end{smallmatrix}\Bigr]$. The length of this path is:
$\intl^1_0 \sqrt{||B_2||^2+||B_3||^2+4t^2||B_4||^2}dt$.

It remains to move from $B(0)$ to $A$.
In this case the straight segment $\overline{B(0),A}$ lies inside $\overline{X_r}$. In total we get:
\begin{align*}
d^{\overline{X_r}}_{in}(A,B)\le \intl^1_0
\sqrt{||B_2||^2+||B_3||^2+4t^2||B_4||^2}dt+||A_1-B_1||.
\end{align*}
Now we use the bounds 
\begin{align*}
\intl^1_0 \sqrt{||B_2||^2+||B_3||^2+4t^2||B_4||^2}dt<
2\sqrt{||B_2||^2+||B_3||^2+||B_4||^2}
\end{align*}
 and $x+y\le \sqrt{2(x^2+y^2)}$  to get:
\begin{align*}
d^{\overline{X_r}}_{in}(A,B) &<2
\sqrt{||B_2||^2+||B_3||^2+||B_4||^2}+||A_1-B_1||\le\\  
&\le 2\sqrt{2}
\sqrt{||A_1-B_1||^2+||B_2||^2+||B_3||^2+||B_4||^2} =2\sqrt{2} \cdot
d_{out}(A,B).
\end{align*}
\end{proof}

\begin{Remark} The constant $2\sqrt{2}$ is certainly not the best one. For example, for $X=Mat_{m\times n}(\K)$ one can prove
$d_{in}^{(\overline{X_r})}(A,B)\le \sqrt{2}d_{out}(A,B)$ by first going along the straight segment $\Bigl[\begin{smallmatrix} B_1&tB_2\\B_3&tB_4\end{smallmatrix}\Bigr]$,
thus bringing $B$ to the form $\Bigl[\begin{smallmatrix}
  B_1&\zero\\B_3&\zero\end{smallmatrix}\Bigr]$, and then going along
the straight segment
$\Bigl[\begin{smallmatrix}
  tA_1+(1-t)B_1&\zero\\(1-t)B_3&\zero\end{smallmatrix}\Bigr]$.

Probably one can get even better bounds by using the appropriate metric on the
Grassmanians of linear subspaces, $Gr(\K^{m-r},\K^m)$, $Gr(\K^{n-r},\K^n)$ or the Stiefel manifolds.
\end{Remark}

\subsection{Lipschitz normality for connected components of $X_r$}
\begin{Theorem}
Let $\K\in\R,\C$ and $X$ be one of the spaces $Mat_{m\times n}(\K)$, $Mat^{sym}_{n\times n}(\K)$, $Mat^{skew-sym}_{n\times n}(\K)$.
Suppose $A,B$ belong to the same connected component of $X_r$, for some $r\le m$.
Then  $\frac{d_{in}^{X_r}(A,B)}{2\sqrt{2}} \le d_{out}(A,B)\le d^{X_r}_{in}(A,B)$.
\end{Theorem}
\begin{proof}
The inequality on the right is obvious, we prove the one on the left.

{\bf Step 1.} (Reduction to the case of $X_n$.)
As in the proof for $\overline{X_r}$ we apply the action of
$U(m)\times U(n)$, or $U(n)$ in the (skew-)symmetric case, to bring
$A$ to the form $\Bigl[\begin{smallmatrix} A_1&\zero\\\zero&\zero\end{smallmatrix}\Bigr]$.  Accordingly $B$ is
brought to $\Bigl[\begin{smallmatrix} B_1& *\\ *& *\end{smallmatrix}\Bigr]$. It might happen that $rank(B_1)<r$. 
To avoid this we can take arbitrarily small but generic deformation of
$B$ inside $X_r$. (For example, apply the group action that adds to 
the first $r$ rows/columns a small but generic linear combination of
all the other rows/columns.) 

Now, as $rank(B_1)=r$, we can take the path $B(t)=\Bigl[\begin{smallmatrix} B_1&t *\\t
*&t^2 *\end{smallmatrix}\Bigr]$, as in the proof for $\overline{X_r}$. As in that proof
the length of this path is less than $2\cdot\sqrt{(\dots)}$.

We arrive to $\Bigl[\begin{smallmatrix} B_1&\zero\\\zero&\zero\end{smallmatrix}\Bigr]$ and it remains to
connect the matrices $\Bigl[\begin{smallmatrix} A_1&\zero\\\zero&\zero\end{smallmatrix}\Bigr]$,  $\Bigl[\begin{smallmatrix}
B_1&\zero\\\zero&\zero\end{smallmatrix}\Bigr]$ inside $X_r$ by a path of the total length
$\le 2d_{out}(A_1,B_1)+\epsilon$. In particular, the initial
question has been reduced to the stratum $X_n$ of square
matrices. Note also: as the path $B(t)$ was fully inside $X_r$, the
points $A_1,B_1$ lie in the same connected component of $X_r$.

{\bf Step 2.}  Let $A,B\in X_n$ where $X=Mat_{n\times n}(\K)$,
$Mat^{sym}_{n\times n}(\K)$ or $Mat^{skew-sym}_{n\times n}(\K)$. 
(For skew-symmetric matrices this implies: $n$ is even.)

\underline{Let $\K=\C$} then all the strata are connected. Consider
the straight segment $[A,B]\subset X$. Its endpoints lie in $X_n$, thus,
by algebraicity of the strata, it intersects $\overline{X_{n-1}}$ in a
finite number of points which is at most  $\deg
(\overline{X_{n-1}})$. Now, by the controlled path connectedness
(Lemma \ref{Thm.Path.Connected.Controlled}), we can deform the path
slightly at each of these point to push it into the stratum
$X_n$. Hence we get a path inside $X_n$ of $\text{length}\le
d_{out}(A,B)+\epsilon$. Together with the path $B(t)$ of step 1 this
finishes the proof.

\

\underline{Suppose $\K=\R$,} let $\mathcal{U}\subset X_n$ be the prescribed
connected component. We construct the needed path from $A$ to $B$
inside $\mathcal{U}$.

{\em The idea of construction.} In the case of $\overline{X_r}$ the straight edge $[A,B]$ was replaced by a straight edge $[A,B(0)]$ and an
algebraic curve from $B(0)$ to $B=B(1)$, such that
\begin{align*}
length\big(A,B(0)\big)+length\big(B(0),B(1)\big)\le 2\sqrt{2}
d_{out}\big(A,B\big), 
\end{align*}
see the proof of Theorem \ref{Thm.Lipshitz.Normality.Closures.of.Strata}. For $\mathcal{U}$ we use the same idea, but we need to split into more paths
to stay inside $\mathcal{U}$. In this way we produce several straight edges, $[A,A_1]$, $[A_1,A_2]$,\dots, $[A_{k-1},A_k],[A_k,B_k]$, and algebraic curves, $(B_k,B_{k-1})$,
$(B_{k-1},B_{k-2})$,\dots, $(B_1,B)$ such that
\begin{align*}
length[A,A_1]&+\cdots+length[A_{k-1},A_k]+length[A_k,B_k]+\\
&+length(B_k,B_{k-1})+\cdots+length(B_1,B)< 2\sqrt{2} d_{out}(A,B)+\epsilon.
\end{align*}
For $X=Mat_{n\times n}(\R)$ or $Mat^{skew-sym}_{n\times n}(\R)$ it is enough to take $k=1$, but for $Mat^{sym}_{n\times n}(\R)$ the number $k$ can
be $\lfloor \frac{n}{2}\rfloor$.
All these paths lie in $\overline{\mathcal{U}}$ and each of them has some points in $\mathcal{U}$, thus (by algebraicity of $\overline{X_r}$) each of the paths lies in $\mathcal{U}$,
except for a finite number of points. At each such point we use the controlled-path-connectedness,
lemma \ref{Thm.Path.Connected.Controlled}, to (slightly) deform the path into $\mathcal{U}$.

\

{\em The construction.}
Fix $A,B\in\mathcal{U}\subset X_n$. The edge $[A,B]$ does not necessarily lie inside $\overline{\mathcal{U}}$, thus (unlike the case $\K=\C$) it cannot be pushed back into $\mathcal{U}$ by
a small deformation. Split the edge $[A,B]$ into the intervals $[A,A_1)$, $[A_1,B_1]$, $(B_1,B]$, where $[A,A_1)\subset\mathcal{U}$,
$(B_1,B]\subset\mathcal{U}$ and $A_1,B_1\in\overline{\mathcal{U}}\setminus\mathcal{U}$. Thus $A_1,B_1\in \overline{X_{n-1}}$, and (after a small-but-generic deformation of $A,B$ inside $\mathcal{U}$)
we can assume $A_1,B_1\in X_{n-1}$. (In the case of $X=Mat^{skew-sym}_{n\times n}(\R)$ the rank drops by two, thus $A_1,B_1\in X_{n-2}$.)
As in the case of $\overline{X_r}$, we can assume (using the $O(n)\times O(n)$ action)
$A_1=\Bigl[\begin{smallmatrix} \widetilde{A}_1&\zero\\\zero&\zero\end{smallmatrix}\Bigr]$, where $\widetilde{A}_1$ is invertible. As in the case of $\overline{X_r}$, we take the algebraic curve
$B_1(t)=\Bigl[\begin{smallmatrix} \widetilde{B}_1& t*\\ t*& t^2*\end{smallmatrix}\Bigr]$, for $t\in[0,1]$. And we can assume $\widetilde{B}_1$ invertible, so this curve lies inside $X_{n-1}$.
(For skew-symmetric matrices the curve lies inside $X_{n-2}$.)

It remains to connect $A_1$ to $B_1(0)$, inside $\overline{\mathcal{U}}$, and to (slightly) deform this path into $\mathcal{U}$.

\begin{itemize}
\item {\em \underline{The case $\mathcal{U}=\Gl^+_n(\R)\subset X=Mat_{n\times n}(\R)$.}} Take the path
\[\begin{bmatrix} t\widetilde{A}_1+(1-t)\widetilde{B}_1&\zero\\\zero&\epsilon(t)\end{bmatrix},
\]
with a continuous function $[0,1]\stackrel{\epsilon(t)}{\to}\R$ that satisfies:
\begin{align*}
&\epsilon(t)\cdot det\Big(t\widetilde{A}_1+(1-t)\widetilde{B}_1\Big)\ge0,\\ &\epsilon(t)=0\ \rm{iff}\ det\Big(t\widetilde{A}_1+(1-t)\widetilde{B}_1\Big)=0\\ &\rm{and} \
|\epsilon(t)|\ll1 \ \text{for any} \ t.
\end{align*}
This path lies inside $\mathcal{U}$ except for a finite number of points, where $det\big(t\widetilde{A}_1+(1-t)\widetilde{B}_1\big)=0$. Now, by
the controlled path-connectedness,  lemma \ref{Thm.Path.Connected.Controlled}, we can deform the path slightly at each of these points into $\mathcal{U}$.
Thus we have
connected (the small deformations of) $A_1,B_1(0)$, inside $\mathcal{U}$, by a path of total length at most $d_{out}(A_1,B_1(0))+\epsilon$.
Together with $B(t)$ this provides the needed path from $A$ to $B$ inside $\mathcal{U}$.

The case $\mathcal{U}=\Gl^-_n(\R)$ is similar.

\item  {\em \underline{The case $X=Mat^{skew-sym}_{n\times n}(\R)$,}}
  here $\mathcal{U} \subset X_n$ is prescribed by the parity of the negative
  values among $\{\lambda_i\}$, see paragraph \ref{Sec.Connected.Components}. Take the path
\[\begin{bmatrix} t\widetilde{A}_1+(1-t)\widetilde{B}_1&\zero&\zero\\\zero&0&\epsilon(t)\\\zero&-\epsilon(t)&0\end{bmatrix},
\]
where a continuous function $[0,1]\stackrel{\epsilon(t)}{\to}\R$ satisfies:
\begin{align*}
\epsilon(t)=0\ \rm{iff}\ det\Big(t\widetilde{A}_1+(1-t)\widetilde{B}_1\Big)=0,\quad  |\epsilon(t)|\ll1 \ \text{for any} \ t,
\end{align*}
and the sign of $\epsilon(t)$ is chosen in such a way that (whenever $det\Big(t\widetilde{A}_1+(1-t)\widetilde{B}_1\Big)\neq0$) the total number of
negative values among $\{\lambda_i\}$ is the one prescribed by $\mathcal{U}$. This path lies inside $\mathcal{U}$, except for a finite number of points where
$det\Big(t\widetilde{A}_1+(1-t)\widetilde{B}_1\Big)=0$. Now use the controlled path-connectedness and proceed as in the case $\mathcal{U}=\Gl_n^+(\R)$.

\item {\em \underline{The case $X=Mat^{sym}_{n\times n}(\R)$}, here $\mathcal{U}=\mathcal{U}_{n_+,n_-}=$symmetric matrices of signature $(n_+,0,n_-)$.}

Suppose at all the points of the edge $[A_1,B_1]$ holds: $n_+(t)\le n_+$, $n_-(t)\le +n_-$. Then we take the path
\[\begin{bmatrix} t\widetilde{A}_1+(1-t)\widetilde{B}_1&\zero\\\zero&\epsilon(t)\end{bmatrix},
\]
and argue as above.

In general on the edge $[\widetilde{A}_1,\widetilde{B}_1]$ there might occur points where one of the conditions $n_+(t)\le n_+$, $n_-(t)\le n_-$ is violated.
And this cannot be corrected by just one factor of $\epsilon(t)$.
Thus we use the reduction on the size of matrix. Split the edge $[A_1,B_1]$ into $[A_1,A_2)$, $[A_2,B_2]$, $(B_2,B_1]$,
where $[A_1,A_2)\subset\overline{\mathcal{U}}\cap X_{n-1}$,
$[B_1,B_2)\subset\overline{\mathcal{U}}\cap X_{n-1}$, and $A_2,B_2\in\overline{\mathcal{U}}\cap \overline{X_{n-2}}$. Push the paths $[A_1,A_2)$,  $[B_1,B_2)$ slightly into $\mathcal{U}$
by $\epsilon(t)$-addition as before. Apply the general method to the edge $[A_2,B_2]$, i.e. by $O(n)\times O(n)$ bring $A_2$ to the canonical form,
then degenerate the corresponding blocks in $B_2$ to zero-blocks. The curve $B_2(1)\rightsquigarrow B_2(0)$ is pushed into $\mathcal{U}$ by
$\begin{bmatrix} \widetilde{B}_2(t)&&\\&\epsilon_1(t)&0\\&0&\epsilon_2(t)\end{bmatrix}$, where $\epsilon_i(t)$ are small corrections as above. Now repeat the process for the edge $[A_2,B_2(0)]$, etc.

After at most $\lfloor\frac{n}{2}\rfloor$ steps we get to some $A_k,B_k$ of rank $\le \lfloor\frac{n}{2}\rfloor$. For them we can take the "corrected" path
\[\begin{bmatrix} t\widetilde{A}_1+(1-t)\widetilde{B}_1&\\&\epsilon_1(t)\\&&\ddots\\&&&\epsilon_k(t)\end{bmatrix},
\]
that lies in $\mathcal{U}$, except for a finite number of points. Now use the controlled path-connectedness.
\end{itemize}\end{proof}

\subsection{Lipschitz normality for transversal intersections with $\overline{X_r}$}

For more general linear subspaces of the space of matrices we have the
following result.

\begin{Proposition}\label{linearEIDS}
Let $V\subset X = Mat_{m\times n}(\K)$ be a linear subspace. Assume that
$V$ intersects $X_r$ transversely for all $r\neq 0$. Then $Y:= V\cap
\overline{X}_r$ is Lipschitz normally embedded. 
\end{Proposition}

\begin{proof}
First  notice that  the stratification of $\overline{X}_r$ is a
locally Lipschitz stratification, since it is locally analytically
trivial along any stratum. Also by
Theorem \ref{Thm.Lipshitz.Normality.Closures.of.Strata}
 $\overline{X}_r$ is Lipschitz normally embedded.
Since $V$ linear then $Y = V\cap \overline{X}_r$ is a cone over its
link, with vertex $0$.

Since $V$ is transverse to all the strata of $\overline{X}_r$ away from
$0,$ then  the stratification of $\overline{X}_r$ induces a
locally Lipschitz trivial stratification on $Y\setminus\{0\}$.

By Proposition \ref{bilipschizttrivial}   $Y\setminus \{0\}$ is locally
Lipschitz normally embedded. The sphere $S$ is transverse to all the
strata of $Y\setminus\{0\},$  here $S\subset Mat_{m\times n}$ is the
sphere of radius $1$ of real codimension $1$ (i.e.\ if $\K=\R$ then
$S=S^{mn-1}$ and if $\K=\C$ then $S=S^{2mn-1}$). Hence we can again use Proposition
\ref{bilipschizttrivial} to conclude that the link $M:=Y\cap S$ is
locally Lipschitz normally embedded. Then by Proposition \ref{localglobal}
 $M$ is Lipschitz normally embedded, and $Y$ is Lipschitz normally
 embedded by Proposition \ref{cone} since it is a cone over $M$.

\end{proof}

It is not true for all linear subspaces $V$ that $V \cap \overline{X}_r$
is Lipschitz normally embedded, as the next example shows. 

\begin{Example}\label{degenerationofcusps}
Let $V\subset Mat_{3\times 3}(\C)$ be the linear subspace given as the
image of the following map $\morf{F}{\C^3}{Mat_{3\times 3}(\C)}$:
\begin{align*}
F(x,y,z)=\left(
\begin{array}{@{} c c c @{}}
x & 0 & z \\ 
y & x & 0 \\ 
0 & y & x
\end{array}
\right).
\end{align*}
Let $Y:= V \cap \overline X_2,$ where $\overline X_2$ is the set of matrices in $Mat_{3\times 3}(\C)$ with zero determinant, which is Lipschitz normally 
embedded by Theorem \ref{Thm.Lipshitz.Normality.Closures.of.Strata}. Hence
one would expect $Y$  to be a nice space. On the other hand
${Y}=V(x^3-y^2z)$, hence it is a family of cusps degeneration to a
line. But ${Y}$ being
Lipschitz normally embedded would imply that the cusp $x^3-y^2=0$ 
is Lipschitz normally embedded by
Proposition \ref{product}, since each non zero point on the $z$-axis has a
neighbourhood which is a product of the cups and the $z$-axis. But the
cusp is not Lipschitz normally embedded by the work of Pham and
Teissier \cite{phamteissier}. Hence
${Y}$ is not Lipschitz normally embedded.
\end{Example}   

The proof of Proposition \ref{linearEIDS} uses the matrix structure of
$Mat_{m\times n}(\K)$, and the naive generalization to more general
varieties does not hold.
\begin{itemize}
\item  The statement "if $X,Y\subset \R^N$ are two
 manifolds intersecting transversally then $X\cap\overline{Y}$ is
 Lipschitz normally embedded"
 does not hold because of the obvious counterexample:
  $Y=\{z=0\}\subset\R^3$, $\overline{X}=\{y^2=x^3\}\subset\R^3$,
  $X=\overline{X}\setminus\{x=0=y\}$.
\item The statement
"if $X,Y\subset \R^N$ are two
 manifolds intersecting transversally, with $\overline{Y}$ Lipschitz
 normal then $X\cap\overline{Y}$ is Lipschitz normally embedded"
 does not hold either.
  e.g. let $\overline{Y}=\{x^2+y^2=z^k\}\subset\R^3$ and $X=\{x=y\}$.
     Then $X\cap \overline{Y}$ is not Lipschitz normally embedded.
\item Consider the embedding$
\R^2\stackrel{j}{\hookrightarrow}X:=Mat_{2\times 2}(\R)$ by
\begin{align*}
(x,y)\to\begin{pmatrix} y&x^{k-1}\\x&y\end{pmatrix}.
\end{align*}
 Then $j(\R^2)$ intersects
transversally   $X_2$ and $X_1$. But
$j(\R^2)\cap\overline{X_1}\approx\{y^2=x^k\}\subset\R^2$ is not
bilipschitz normal at the origin. So the linearity of $V$ is also
important. We will in Section \ref{secgeneralcase} look more on the
case when $V$ is not linear.
\end{itemize}

\section{Lipschitz normality of collections of affine subspaces in $\R^N$}\label{uppertriangularcase}
Fix a (possibly infinite) collection $\{L_i\}$ of affine subspaces in $\R^N$, of (varying) positive dimensions. The union $\cup L_i$ is not
always Lipschitz normally embedded (of course we assume  $\cup L_i$ to be connected).
\begin{Example} The subset $\{x(y^2-1)=0\}\subset\R^2$ is not Lipschitz normally embedded because
$d_{out}\big((t,1),(t,-1)\big)=2$, while
$d_{in}\big((t,1),(t,-1)\big)=2+2t$. \end{Example} In this example the
collection contains two non-intersecting lines. We prove that in
many cases this is the only obstruction to bieng Lipschitz normally embedded.

\subsection{}
As a preparation we recall the definition of the angle between two (intersecting) affine subspaces $L_i,L_j\subset\R^N$. (All the metrics here are outer.)
\begin{itemize}
\item Suppose the intersection is just one point, $L_i\cap L_j=\{0\}$. Define the angle, $\alpha_{L_i,L_j}\in[0,\frac{\pi}{2}]$, via the theorem of cosines:
\begin{align*}
cos(\alpha_{L_i,L_j})=\underset{\substack{x\in L_i\setminus\{0\}\\y\in
    L_j\setminus\{0\}}}{sup}\frac{d^2(x,0)+d^2(y,0)-d^2(x,y)}{2d(x,0)d(y,0)}.
\end{align*}
If $L_i,L_j$ are lines this gives the classical definition, in particular it is independent on the choice of $x,y$.
\item
If $dim(L_i\cap L_j)>0$ fix a point $0\in L_i\cap L_j$ and take the orthogonal complement at $0$:
$(L_i\cap L_j)\oplus \mathcal{N}=\R^N$. Then we define $\alpha_{L_i,L_j}:=\alpha_{(L_i\cap \mathcal{N}),(L_j\cap \mathcal{N})}$. If $dim(L_i)=dim(L_j)$ and $dim(L_i\cap L_j)=dim(L_i)-1$
then $L_i\cap\mathcal{N}$, $L_j\cap \mathcal{N}$ are lines and we get the classical definition.
\end{itemize}
By its definition $\alpha_{L_i,L_j}\in[0,\frac{\pi}{2}]$.
\begin{Lemma}
If $L_i\not\subseteq L_j$ and $L_j\not\subseteq L_i$ then $\alpha_{L_i,L_j}\neq0$.
\end{Lemma}
\begin{proof}
We can assume $L_i\cap L_j$ is just one point and move this point to the origin.
In the definition of  $cos(\alpha_{L_i,L_j})$ apply the homogeneous scaling of $\R^N$ to get: $0<\epsilon\le |x|,|y|\le1$. As $x\not\in L_j$ and $y\not\in L_i$ the points
$x,0,y$ are not on one line, thus:
$f(x,y)=\frac{d^2(x,0)+d^2(y,0)-d^2(x,y)}{2d(x,0)d(y,0)}<1$. As $f(x,y)$ is a continuous function on the compact domain,
$(L_i\cap \{\epsilon\le |x|\le1\} )\times (L_j\cap \{\epsilon\le |y|\le1\} )$, it attains its maximum. Therefore
$cos(\alpha_{L_i,L_j})=\underset{\substack{x\in L_i\setminus\{0\}\\y\in L_j\setminus\{0\}}}{sup} f(x,y)<1$.
\end{proof}

\subsection{}
Now we use the angle $\alpha_{L_i,L_j}$ to get the optimal Lipschitz constant.
\begin{Proposition}
Let $X=\cup L_i\subset\R^N$ be the union of affine subspaces. Suppose
$L_i\not\subseteq L_j$ for any $i\neq j$ and the subspaces intersect
pairwise, $L_i\cap L_j\neq\varnothing$.
\begin{enumerate}
\item For any $x,y\in X$ holds:
\begin{align*}\frac{d^{(X)}_{in}(x,y)}{d_{out}(x,y)}
\leq \supl_{i \neq j}
\frac{1}{sin(\frac{\alpha_{L_i,L_j}}{2})}.
\end{align*}\label{521}
\item If the collection is finite then the bound is asymptotically
  sharp, i.e. there exist sequences $\{x_n\}$, $\{y_n\}$ satisfying:
\begin{align*}
\frac{d^{(X)}_{in}(x_n,y_n)}{d_{out}(x_n,y_n)}\to\supl_{i \neq j}
\frac{1}{sin( \frac{\alpha_{L_i,L_j}}{2})}.
\end{align*}\label{522}
\end{enumerate}
\end{Proposition}
\begin{proof}
\eqref{521} Let $x\in L_i$, $y\in L_j$, the non-trivial case is $i\neq
j$. Fix some $0\in L_i\cap L_j$ and use the theorem of sines for the
triangle $Conv(x,y,0)$:
$\frac{d(0,x)}{sin(\alpha_y)} = \frac{d(0,y)}{sin(\alpha_x)} =
\frac{d(x,y)}{sin(\alpha_{0})}$. Thenone has: 
\begin{align*}
d^{(X)}_{in}(x,y) &\le
d(x,0)+d(0,y) =\frac{d_{out}(x,y)(sin(\alpha_y)+sin(\alpha_x))}{sin(\alpha_0)}=
\\ &=
d_{out}(x,y)\frac{2sin\frac{\alpha_x+\alpha_y}{2}cos(\frac{\alpha_x-\alpha_y}{2})}{sin(\alpha_0)}\le
2d_{out}(x,y)\frac{cos\frac{\alpha_0}{2}}{sin(\alpha_0)}.
\end{align*}
Finally, $\alpha_0\ge\alpha_{L_i,L_j}$ hence $d^{(X)}_{in}(x,y)\le d_{out}(x,y)\frac{1}{sin\frac{\alpha_{L_i,L_j}}{2}}$. This gives the bound.

\

\eqref{522} To prove the asymptotic sharpness note that for
$|x_n|,|y_n|\to \infty$, the main contribution to
$d^{(X)}_{in}(x_n,y_n)$ comes from the paths inside $L_i,L_j$,
while the possible corrections from other affine spaces become negligible.
\end{proof}

We remark that though the statement does not assume finiteness of the
collection $\{L_i\}$, it is not very useful in the infinite case, as there
$\supl_{i \neq j} \frac{1}{sin(\frac{\alpha_{L_i,L_j}}{2})}$ can easily go to infinity.

In this way one can produce many non-Cohen-Macaulay singularities
which are still Lipschitz normally embedded.

\begin{Example}\label{Ex.Triangular.Lip.Norm.corank=1}
Suppose for some $X\subset Mat_{m\times n}(\K)$ the stratum $\overline{X_r}$ consists of linear subspaces. (They
all intersect as $\zero$ belongs to each of them.) Then  $\overline{X_r}$  is Lipschitz normally embedded. For example let $X$ be the subspace of (upper/lower)
triangular matrices in $Mat_{m\times m}(\R)$, then $\overline{X_{m-1}}$ is Lipschitz normally embedded. As all $L_i$ are orthogonal in this case the optimal
Lipschitz constant is $\sqrt{2}$.
\end{Example}


\section{The case of determinantal singularities}\label{secgeneralcase}

In this section we  discuss  Lipschitz normal embeddings
of determinantal singularities. The spaces of matrices we worked with
in the previous sections can be seen as special cases of determinantal
singularities. In this section we  assume that $X=
Mat_{m\times n}(\K)$, hence $\overline{X}_r$ is the matrices of rank less
than or equal to $r$. One could also work with $Mat_{m\times n}^{sym}(\K)$ or
$Mat_{m\times n}^{skew-sym}(\K)$ but for simplicity we will restrict
our discussion to $Mat_{m\times n}(\K)$. 

Let $\morf{F}{(\K^N,0)}{(Mat_{mn}(\K),0)}$ be an analytic
map germ. Then $Y:=F\inv(\overline{X}_r)$ is a \emph{determinantal variety of
type} $(m,n,r+1)$ if $\codim (Y)=\codim (\overline{X}_r)$, here we
assume that $r<\min\{m,n\}$. Following \'Ebeling and
Guse{\u\i}n-Zade \cite{ebelingguseinzade} a determinantal singularity $Y=F\inv(\overline{X}_r)$ has an \emph{essentially isolated singularity} at the origin
 (EIDS for
short) if there is a  neighbourhood $U$ of the origin, such
that $F|_{U\setminus \{0\}}$ is transversal to the stratification  of
$X_r.$ That is, for every $x \in U\setminus \{0\},$   rank of $F(x)=s,
\, 0\leq s \leq r,$ then $F$ is transversal to $X_s$ at $x.$  Any ICIS
is an EIDS of type $(m,1,1)$.

With the  notion of determinantal singularities
Proposition \ref{linearEIDS} becomes the following:
\begin{Theorem}\label{EIDS2}
Let $Y$ is an EIDS defined by a linear map-germ
$\morf{F}{\K^N}{Mat_{m\times n}}$, then $Y$ is Lipschitz normally embedded. 
\end{Theorem}
\begin{proof}
If $F$ is injective then this is just a reformulation of Proposition
\ref{linearEIDS}. 
So assume that $F$ is not injective, then we can decompose $\K^N$ as
$\ker (F)\oplus V$, where $F$ induces an isomorphism from $V$ to $\im
(F)$. Hence $Y=F\inv(\overline{X_r})$ is isomorphic to $\ker (F)\oplus
(\im (F)\cap \overline{X_r}).$  Now $\ker(F)$ is a linear space and
hence Lipschitz normally embedded and $\im (F)\cap \overline{X_r}$
is Lipschitz normally embedded by Proposition \ref{linearEIDS}, hence $Y$ is
Lipschitz normally embedded by Proposition \ref{product}. 
\end{proof}

We can make a more general statement in Theorem \ref{EIDS2}. Take the group $G= \mathcal R \times \mathcal H$ acting on the space of map-germs
    $F: (\K^N,0) \to ( M_{m,n}(\K),0)$ where $\mathcal R$ is the group
    of germs of diffeomorphisms in
    $(\K^N,0)$ and $\mathcal H$ is the group $GL_m(\mathcal O_N) \times GL_n(\mathcal O_N)$, given
    by invertible matrices with entries in $(\mathcal O_N,0)$ (see for
    instance Fr\"uhbis Kr\"uger and Neumer's
    \cite{fruhbiskrugerneumer}). As a consequence of Theorem
    \ref{EIDS2} and Lemma \ref{changediffeo} we can state the
    following:
\begin{Corollary} If $\morf{F}{(\K^N,0)}{(Mat_{m\times n}(\K),0)}$ is
  $G$-equivalent to a linear EIDS, then Thoerem \ref{EIDS2} holds.
\end{Corollary}

Whether a determinantal singularity is Lipschitz normally embedded is
in general a more difficult question than for singularities in the
space of matrices. One cannot in general expect a determinantal
singularity to be Lipschitz normally embedded, the easiest way to see
this is to note that all ICIS are determinantal, and that there are
many ICIS that are not Lipschitz normally embedded. For example
among the simple complex surface singularities $A_n$, $D_n$, $E_6$, $E_7$ and
$E_8$ only the $A_n$'s are Lipschitz normally embedded. Since the
structure of determinantal singularities does not give us any new
tools to study ICIS, we will probably not be able to say when an ICIS
is Lipschitz normally embedded. Since $F\inv(\overline{X}_0)$
is often an ICIS, we probably have to assume it is Lipschitz normally
embedded to say 
anything about whether $F\inv(\overline{X}_t)$ is Lipschitz normally
embedded. But before we discuss such assumption further, we will see
what went wrong in our Example \ref{degenerationofcusps} and give
some more examples of determinantal singularities that are Lipschitz
normally embedded and some that are not.

In Example \ref{degenerationofcusps},
$Y_0:= F\inv(\overline{X}_0)$ is a point and $Y_1:= F\inv(\overline{X}_1)$ is
a line, so both $Y_0$ and $Y_1$ are Lipschitz normally embedded. So it
does not in general follows that if $Y_i$ is Lipschitz normally
embedded then $Y_{i+1}$ is. Now the singularity in Example
\ref{degenerationofcusps} is not an EIDS since $F^{-1}(\overline X_1)$
does not have the expected dimension (the expected dimension is
$-1$). In the next example we will see that EIDS is not enough either.

\begin{Example}[Simple Cohen-Macaulay codimensional 2 surface singularities]\label{scmc2ss}
In \cite{fruhbiskrugerneumer} Fr\"uhbis-Kr\"uger and Neumer classify
simple complex Cohen-Macaulay codimension 2 singularities. They are all EIDS of
type $(3,2,2)$, and the surfaces correspond to the
rational triple points classified by Tjurina \cite{tjurina}. We will
look closer at two of such families. First we have the family
given by the matrices:
\begin{align*}
\left(
\begin{matrix}
z & y+w^l & w^m \\ 
w^k & y & x
\end{matrix}
\right).
\end{align*} 
This family corresponds to the family of triple points in
\cite{tjurina} called $A_{k-1,l-1,m-1}$. Tjurina shows that the dual
resolution graph of their minimal resolution are:
$$
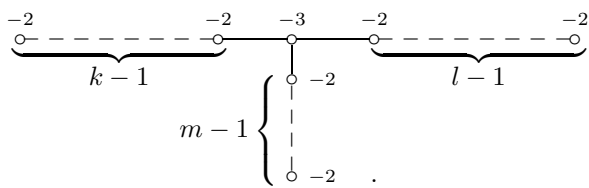
\xymatrix@R=6pt@C=24pt@M=0pt@W=0pt@H=0pt{
&&\\
\overtag{\Circ}{-2}{8pt}\dashto[rr] &
{\hbox to 0pt{\hss$\underbrace{\hbox to 80pt{}}$\hss}}&
\overtag{\Circ}{-2}{8pt}\lineto[r] &
\overtag{\Circ}{-3}{8pt}\lineto[r]\lineto[d] &
\overtag{\Circ}{-2}{8pt}\dashto[rr] &
{\hbox to 0pt{\hss$\underbrace{\hbox to 80pt{}}$\hss}}&
\overtag{\Circ}{-2}{8pt}\\
&{k-1}&& \righttag{\Circ}{-2}{8pt}\dashto[dddd] &&{l-1}\\
&&&&\\
&&&&\blefttag{\quad}{m-1\begin{cases} \quad \\
    \ \\ \ \end{cases}}{10pt} & \\
&&&&\\
&&& \righttag{\Circ}{-2}{8pt} & .\\
&&}$$ 
Using Remark 2.3 of \cite{spivakovsky} we see that these singularities
are minimal, and hence by the result of \cite{normallyembedded} we get
that they are Lipschitz normally embedded.

The second family is given by the matrices:
\begin{align*}
\left(
\begin{matrix}
z & y+w^l & xw \\ 
w^k & x & y
\end{matrix}
\right).
\end{align*} 
Tjurina calls this family $B_{2l,k-1}$ and give the dual resolution
graphs of their minimal resolutions as:
$$
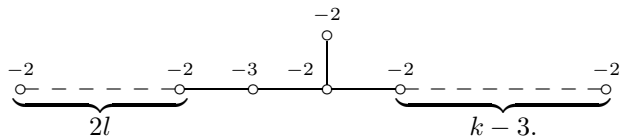
\xymatrix@R=6pt@C=24pt@M=0pt@W=0pt@H=0pt{
&&&& \overtag{\Circ}{-2}{8pt} &\\
&&&&&\\
\overtag{\Circ}{-2}{8pt}\dashto[rr] &
{\hbox to 0pt{\hss$\underbrace{\hbox to 65pt{}}$\hss}}&
\overtag{\Circ}{-2}{8pt}\lineto[r] &
\overtag{\Circ}{-3 \hspace{7pt}}{8pt}\lineto[r] &
\overtag{\Circ}{-2 \hspace{20pt}}{8pt}\lineto[r] \lineto[uu]&
\overtag{\Circ}{-2}{8pt}\dashto[rr]
&{\hbox to 0pt{\hss$\underbrace{\hbox to 80pt{}}$\hss}}&
\overtag{\Circ}{-2}{8pt}\\
&2l&&& &&k-3.& \\
&&}$$ 
Following Spivakovsky this is not a minimal singularity, and since it
is rational according to Tjurina it is not Lipschitz normally embedded
by the result of \cite{normallyembedded}.

These two families do not look very different but one is Lipschitz
normally embedded and the other is not. We can do the same for all simple
Cohen-Macaulay codimension 2 surfaces, and using the results in
\cite{normallyembedded}, that rational surface singularities are
Lipschitz normally embedded if and only if they are minimal, we get
that only the family $A_{l,k,m}$ is Lipschitz normally
embedded. This is similar to the case of codimension 1, since only the
$A_n$ singularities are Lipschitz normally embedded among the simple
singularities. 
\end{Example}

So as we see in Example \ref{scmc2ss} being an EIDS with singular
set Lipschitz normally embedded, is not enough to ensure the variety is
Lipschitz normally embedded. One should notice that the varieties
in Example \ref{degenerationofcusps} and \ref{scmc2ss} are both
defined by maps $\morf{F}{\C^N}{Mat_{m\times n}}$ where $N<mn$. This
means that one should think of the singularity as a section of
$\overline{X}_t$, but being a subspace of a Lipschitz normally embedded space
does not imply the Lipschitz normally embedded condition. 
If $N\geq mn$ then 
one can think about the singularity being a fibration over $\overline{X}_t$,
and as we saw in Proposition \ref{product} products of Lipschitz
normally embedded spaces are Lipschitz normally embedded. Now in this
case $Y_0=F\inv(\overline{X}_0)$ is ICIS if $Y$ is an EIDS, which means that we
probably can not say anything general about whether it is Lipschitz
normally embedded or not. So natural assumptions would be to assume
that $Y$ is an EIDS and that $Y_0$ is Lipschitz normally embedded.

\section*{Acknowledgements}
The second and third author would like to thank Walter Neumann and
Anne Pichon for 
first letting us know 
about Asuf Shachar's question on Mathoverlfow.org, and for helpful
comments about the manuscript. We would also like to thank Lev Birbrair for
sending us an early version of the paper by Katz, Katz, Kerner and
Liokumovich  \cite{kerneretc}, and encouraging us to work on
the problem. We also thank Nguyen Xuan Viet Nhan for help with the proof of
Proposition \ref{bilipschizttrivial}. The first author was supported by the grant FP7-People-MCA-CIG, 334347, the second author was supported by FAPESP grant
2015/08026-4 and the third author was partially supported by FAPESP
grant 2014/00304-2 and CNPq grant 306306/2015-8.

\bibliography{general}

\def\cprime{$'$}
\begin{thebibliography}{NBOOT13}

\bibitem[ACGH85]{arbarellocornalbagriffiths}
Enrico Arbarello, Maurizio Cornalba, Phillip~A. Griffiths, and Joe Harris.
\newblock {\em Geometry of algebraic curves. {V}ol. {I}}, volume 267 of {\em
  Grundlehren der Mathematischen Wissenschaften [Fundamental Principles of
  Mathematical Sciences]}.
\newblock Springer-Verlag, New York, 1985.

\bibitem[BF08]{birbrairfernandes}
Lev Birbrair and Alexandre Fernandes.
\newblock Inner metric geometry of complex algebraic surfaces with isolated
  singularities.
\newblock {\em Comm. Pure Appl. Math.}, 61(11):1483--1494, 2008.

\bibitem[BFLS16]{lipschitzregularity}
Lev Birbrair, Alexandre Fernandes, D{\~u}ng~Tr{\'a}ng L{\^e}, and J.~Edson
  Sampaio.
\newblock Lipschitz regular complex algebraic sets are smooth.
\newblock {\em Proc. Amer. Math. Soc.}, 144(3):983--987, 2016.

\bibitem[BNP14]{thickthin}
Lev Birbrair, Walter~D. Neumann, and Anne Pichon.
\newblock The thick-thin decomposition and the bilipschitz classification of
  normal surface singularities.
\newblock {\em Acta Math.}, 212(2):199--256, 2014.

\bibitem[DP14]{damonpike}
James Damon and Brian Pike.
\newblock Solvable groups, free divisors and nonisolated matrix singularities
  {II}: {V}anishing topology.
\newblock {\em Geom. Topol.}, 18(2):911--962, 2014.

\bibitem[Fer03]{fernandesplanecurve}
Alexandre Fernandes.
\newblock Topological equivalence of complex curves and bi-{L}ipschitz
  homeomorphisms.
\newblock {\em Michigan Math. J.}, 51(3):593--606, 2003.

\bibitem[FKN10]{fruhbiskrugerneumer}
Anne Fr{\"u}hbis-Kr{\"u}ger and Alexander Neumer.
\newblock Simple {C}ohen-{M}acaulay codimension 2 singularities.
\newblock {\em Comm. Algebra}, 38(2):454--495, 2010.

\bibitem[FZ15]{fruhbiskrugerzach}
Anne {Fruehbis-Krueger} and Matthias {Zach}.
\newblock {On the Vanishing Topology of Isolated Cohen-Macaulay Codimension 2
  Singularities}.
\newblock {\em ArXiv e-prints}, 1501.01915, January 2015.

\bibitem[GR15]{gaffenyrangachev}
Terence {Gaffney} and Antoni {Rangachev}.
\newblock {Pairs of modules and determinantal isolated singularities}.
\newblock {\em ArXiv e-prints}, 1501.00201, December 2015.

\bibitem[GZ{\`E}09]{ebelingguseinzade}
Sabir~M. Guse{\u\i}n-Zade and Wolfgang {\`E}beling.
\newblock On the indices of 1-forms on determinantal singularities.
\newblock {\em Tr. Mat. Inst. Steklova}, 267(Osobennosti i
  Prilozheniya):119--131, 2009.

\bibitem[KKKL17]{kerneretc}
Karin~U. {Katz}, Mikhail~G. {Katz}, Dmitry {Kerner}, and Yevgeny.
  {Liokumovich}.
\newblock Determinantal variety and normal embedding.
\newblock {\em Journal of Topology and Analysis}, 09(1):1--8, 2017.

\bibitem[MP98]{macphersonprocesi}
Robert MacPherson and Claudio Procesi.
\newblock Making conical compactifications wonderful.
\newblock {\em Selecta Math. (N.S.)}, 4(1):125--139, 1998.

\bibitem[NBOOT13]{NunoOreficeOkamotoTomazella}
Juan~J. Nu{\~n}o-Ballesteros, Bruna Or{\'e}fice-Okamoto, and Jo{\~a}o~N.
  Tomazella.
\newblock The vanishing {E}uler characteristic of an isolated determinantal
  singularity.
\newblock {\em Israel J. Math.}, 197(1):475--495, 2013.

\bibitem[NP14a]{NeumannPichonBilpischitzGeometryOfCurves}
Walter~D. Neumann and Anne Pichon.
\newblock Lipschitz geometry of complex curves.
\newblock {\em J. Singul.}, 10:225--234, 2014.

\bibitem[NP14b]{zariski}
Walter~D. Neumann and Anne Pichon.
\newblock Lipschitz geometry of complex surfaces: analytic invariants and
  equisingularity.
\newblock {\em arxiv:1211.4897}, 2014.

\bibitem[NPP15]{normallyembedded}
Walter~D. Neumann, Helge~M\o{}ller Pedersen, and Anne Pichon.
\newblock Minimal surface singularities are {L}ipschitz normally embedded.
\newblock {\em arxiv:1503.03301}, 2015.

\bibitem[Par93]{parusinski}
Adam Parusi\'nski.
\newblock Lipschitz stratification.
\newblock In {\em Global analysis in modern mathematics ({O}rono, {ME}, 1991;
  {W}altham, {MA}, 1992)}, pages 73--89. Publish or Perish, Houston, TX, 1993.

\bibitem[PR16]{lnedeterminantalsing}
Helge~M\o{}ller Pedersen and Maria Aparecida~Soaras Ruas.
\newblock {Lipschitz Normal Embeddings and Determinantal Singularities}.
\newblock {\em ArXiv e-prints}, July 2016.

\bibitem[PT69]{phamteissier}
Fr{\'e}d{\'e}ric Pham and Bernard Teissier.
\newblock {Fractions Lipschitziennes d'une alg{\`e}bre analytique complexe et
  saturation de Zariski, par Fr{\'e}d{\'e}ric Pham et Bernard Teissier}.
\newblock 42 pages. Ce travail est la base de l'expos{\'e} de Fr{\'e}d{\'e}ric
  Pham au Congr{\`e}s International des Math{\'e}maticiens, Nice 1970., June
  1969.

\bibitem[Spi90]{spivakovsky}
Mark Spivakovsky.
\newblock Sandwiched singularities and desingularization of surfaces by
  normalized {N}ash transformations.
\newblock {\em Ann. of Math. (2)}, 131(3):411--491, 1990.

\bibitem[SRDSP14]{cedinhamiriam}
Maria~Aparecida Soares~Ruas and Miriam Da~Silva~Pereira.
\newblock Codimension two determinantal varieties with isolated singularities.
\newblock {\em Math. Scand.}, 115(2):161--172, 2014.

\bibitem[Tju68]{tjurina}
Galina~N. Tjurina.
\newblock Absolute isolation of rational singularities, and triple rational
  points.
\newblock {\em Funkcional. Anal. i Prilo\v zen.}, 2(4):70--81, 1968.

\end{thebibliography}

\end{document}